\documentclass[12pt]{amsart}
\usepackage{amssymb}
\usepackage{graphicx}
\newtheorem{Def}{Definition}
\newtheorem{Lem}{Lemma}

\newtheorem{Thm}{Theorem}
\newtheorem{Cor}{Corollary}
\newtheorem{Rem}{Remark}
\newenvironment{Pf}{ Proof.}{\(\square\)}

\title[On the extremal compatible linear connection...]{On the extremal compatible linear connection of a generalized Berwald manifold}
\author{Csaba Vincze}
\address{Inst. of Math., Univ. of Debrecen \\
H-4002 Debrecen, P.O.Box 400 \\
Hungary}
\email{csvincze@science.unideb.hu}
\keywords{Finsler spaces, Generalized Berwalds spaces, Intrinsic Geometry}
\subjclass{53C60, 58B20}
\begin{document}
\begin{abstract}
Generalized Berwald manifolds are Finsler manifolds admitting linear connections such that the parallel transports preserve the Finslerian length of tangent vectors (compatibi\-li\-ty condition). By the fundamental result of the theory \cite{V5} such a linear connection must be metrical with respect to the averaged Riemannian metric given by integration of the Riemann-Finsler metric on the indicatrix hypersurfaces. Therefore the linear connection (preserving the Finslerian length of tangent vectors) is uniquely determined by its torsion. If the torsion is zero then we have a classical Berwald manifold. Otherwise, the torsion is a strange data we need to express in terms of the intrinsic quantities of the Finsler manifold. In the paper we consider the extremal compatible linear connection of a generalized Berwald manifold by minimizing the pointwise length of its torsion tensor. It is a conditional extremum problem involving functions defined on a local neighbourhood of the tangent manifold. In case of a given point of the manifold, the reference element method provides that the number of the Lagrange multipliers equals to the number of the equations providing the compatibility of the linear connection to the Finslerian metric. Therefore the solution of the conditional extremum problem with a reference element  can be expressed in terms of the canonical data. The solution of the conditional extremum problem independently of the reference elements can be constructed algorithmically at each point of the manifold. The pointwise solutions constitute a section of the torsion tensor bundle for testing the compatibility of the corresponding linear connection to the Finslerian metric. In other words, we have an intrinsic algorithm to check the existence of compatible linear connections on a Finsler manifold because it is equivalent to the existence of the extremal compatible linear connection.
\end{abstract}

\maketitle
\footnotetext[1]{Cs. Vincze is supported by the EFOP-3.6.1-16-2016-00022 project. The project is co-financed by the European Union and the European Social Fund. The work is also supported by TKA-DAAD 307818. }

\section*{Introduction}

The notion of the generalized Berwald manifolds goes back to V. Wagner \cite{Wag1}. They are Finsler manifolds admitting linear connections such that the parallel transports preserve the Finslerian length of tangent vectors (compatibility condition). We are interested in the uni\-city of the compatible linear connection and its expression in terms of the canonical data of the Finsler manifold (intrinsic characterization). If the torsion is zero (classical Berwald manifolds), the intrinsic characterization is the vanishing of the mixed curvature tensor of the canonical horizontal distribution. The problem of the intrinsic characterization is solved in the more general case of Finsler manifolds admitting semi-symmetric compatible linear connections \cite{V10}, see also \cite{V11}. We also have a unicity statement because the torsion tensor of the semi-symmetric compatible linear connection can be explicitly expressed in terms of metrics and differential forms given by averaging. Especially, the integration of the Riemann-Finsler metric on the indicatrix hypersurfaces (the so-called averaged Riemannian metric) provides a  Riemannian environment for the investigations. The fundamental result of the theory \cite{V5} states that a linear connection satisfying the compatibility condition must be metrical with respect to the averaged Riemannian metric. Therefore the compatible linear connection is uniquely determined by its torsion tensor. Unfortunately, the unicity statement for the compatible linear connection of a generalized Berwald manifold is false in general \cite{Vin1}. To avoid the difficulties coming from different possible solutions, the idea is to look for the extremal solution in some sense: the extremal compatible linear connection of a generalized Berwald manifold keeps its torsion as close to the zero as possible. It is a conditional extremum problem involving functions defined on a local neighbourhood of the tangent manifold. In case of a given point of the manifold, the reference element method provides that the number of the Lagrange multipliers equals to the number of the equations providing the compatibility of the linear connection to the Finslerian metric. Therefore the solution of the conditional extremum problem with a reference element can be expressed in terms of the canonical data. The solution of the conditional extremum problem independently of the reference elements can be constructed algorithmically at each point of the manifold. The pointwise solutions constitute a continuous section of the torsion tensor bundle. The continuity of the components of the torsion tensor implies the continuity of the connection parameters. Using parallel translations with respect to such a connection we can conclude that the Finsler metric is monochromatic. By the fundamental result of the theory \cite{BM} it is sufficient and necessary for a Finsler metric to be a generalized Berwald metric. Therefore we have an intrinsic algorithm to check the existence of compatible linear connections on a Finsler manifold because it is equivalent to the existence of the extremal compatible linear connection.

\section{Notations and terminology}

Let $M$ be a differentiable manifold with local coordinates $u^1, \ldots, u^n.$ The induced coordinate system of the tangent manifold $TM$ consists of the functions $x^1, \ldots, x^n$ and $y^1, \ldots, y^n$. For any $v\in T_pM$, $x^i(v):=u^i\circ \pi (v)=p$ and $y^i(v)=v(u^i)$, where $i=1, \ldots, n$ and $\pi \colon TM \to M$ is the canonical projection. 
A Finsler metric is a continuous function $F\colon TM\to \mathbb{R}$ satisfying the following conditions:  $\displaystyle{F}$ is smooth on the complement of the zero section (regularity), $\displaystyle{F(tv)=tF(v)}$ for all $\displaystyle{t> 0}$ (positive homogeneity) and the Hessian 
$$g_{ij}=\frac{\partial^2 E}{\partial y^i \partial y^j}$$
of the energy function $E=F^2/2$ is positive definite at all nonzero elements $\displaystyle{v\in T_pM}$ (strong convexity), $p\in M$. The so-called \emph{Riemann-Finsler metric} $g$ is constituted by the components $g_{ij}$. It is defined on the complement of the zero section. The Riemann-Finsler metric makes the complement of the origin a Riemannian manifold in each tangent space. The canonical objects are the {\emph {volume form}} $\displaystyle{d\mu=\sqrt{\det g_{ij}}\ dy^1\wedge \ldots \wedge dy^n}$,
the \emph {Liouville vector field} $\displaystyle{C:=y^1\partial /\partial y^1 +\ldots +y^n\partial / \partial y^n}$ and the {\emph {induced volume form}}
$$\mu=\sqrt{\det g_{ij}}\ \sum_{i=1}^n (-1)^{i-1} \frac{y^i}{F} dy^1\wedge\ldots\wedge dy^{i-1}\wedge dy^{i+1}\ldots \wedge dy^n$$
 on the indicatrix hypersurface $\displaystyle{\partial K_p:=F^{-1}(1)\cap T_pM\ \  (p\in M)}$. The averaged Riemannian metric is defined by 
\begin{equation}
\label{averagemetric1}
\gamma_p (v,w):=\int_{\partial K_p} g(v, w)\, \mu=v^i w^j \int_{\partial K_p} g_{ij}\, \mu \ \ (v, w\in T_p M, p\in U).
\end{equation}

\begin{Def} A linear connection $\nabla$ on the base manifold $M$ is called \emph{compatible} to the Finslerian metric if the parallel transports with respect to $\nabla$ preserve the Finslerian length of tangent vectors. Finsler manifolds admitting compatible linear connections are called generalized Berwald manifolds.
\end{Def}

Suppose that the parallel transports with respect to $\nabla$ (a linear connection on the base manifold) preserve the Finslerian length of tangent vectors and let $X$ be a parallel vector field along the curve $c\colon [0,1]\to M$. We have that
\begin{equation}
\label{eq:5}
(F \circ X)'=(x^k\circ X)'{\frac{\partial F}{\partial x^k}}\circ X+(y^k \circ X)'{\frac{\partial F}{\partial y^k}}\circ X
\end{equation}
where $\displaystyle{(x^k\circ X)'={c^k}'}$ and $\displaystyle{{X^k}'=-{c^i}'  X^j  \Gamma_{ij}^k\circ c}$ 
because of the differential equation for parallel vector fields. Therefore 
\begin{equation}
\label{eq:55}
(F \circ X)'={c^i}'\bigg(\frac{\partial F}{\partial x^i}-y^j {\Gamma}_{ij}^{k}\circ \pi \frac{\partial F}{\partial y^k}\bigg)\circ X.
\end{equation}
This means that the parallel transports with respect to $\nabla$ preserve the Finslerian length of tangent vectors (compatibility condition) if and only if
\begin{equation}
\label{cond1}
\frac{\partial F}{\partial x^i}-y^j {\Gamma}^k_{ij}\circ \pi \frac{\partial F}{\partial y^k}=0,
\end{equation}
where $i=1, \ldots,n$. The vector fields of type
\begin{equation}
\label{eq:6}
X_i^{h}:=\frac{\partial}{\partial x^i}-y^j {\Gamma}^k_{ij}\circ \pi \frac{\partial}{\partial y^k}
\end{equation}
span the horizontal distribution belonging to $\nabla$. In a similar way, we can introduce the horizontal vector fields $X_i^{h^*}$ ($i=1, \ldots, n$) with respect to the L\'{e}vi-Civita connection of the averaged Riemannian metric $F^*(v):=\sqrt{\gamma_p(v,v)}.$ 

\begin{Thm}
\label{heritage} \emph{\cite{V5}} If a linear connection on the base manifold is compatible to the Finslerian metric function then it must be metrical with respect to the averaged Riemannian metric.
\end{Thm}

In what follows we are going to substitute the connection parameters with the components of the torsion tensor in the equations of the compatibility condition (\ref{cond1}). Since the torsion tensor bundle can be equipped with a Riemannian metric in a natural way, we can measure the length of the torsion to formulate an extremum problem for the  compatible linear connection keeping its torsion as close to the origin as possible. 

\section{The extremal compatible linear connection of a generalized Berwald manifold}

Let $F$ be the Finslerian metric of a connected generalized Berwald manifold and suppose that $\nabla$ is a compatible linear connection. Taking vector fields with pairwise vanishing Lie brackets on a local neighbourhood of the base manifold, the Christoffel process implies that
$$X\gamma(Y,Z)+Y\gamma(X,Z)-Z\gamma(X,Y)=$$
$$2\gamma(\nabla_X Y, Z)+\gamma(X, T(Y,Z))+\gamma(Y, T(X,Z))-\gamma(Z, T(X,Y))$$
and, consequently, 
\begin{equation}
\label{Cproc}
\gamma(\nabla^*_X Y,Z)=\gamma(\nabla_X Y, Z)+\frac{1}{2}\left(\gamma(X, T(Y,Z))+\gamma(Y, T(X,Z))-\gamma(Z, T(X,Y))\right),
\end{equation}
where $\nabla^*$ denotes the L\'{e}vi-Civita connection of the averaged Riemannian metric $\gamma$ and $T$ is the torsion tensor of $\nabla$. In terms of the connection parameters
\begin{equation}
\label{torsion}
\Gamma_{ij}^r=\Gamma_{ij}^{*r}-\frac{1}{2}\left(T^{l}_{jk}\gamma^{kr}\gamma_{il}+T^{l}_{ik}\gamma^{kr}\gamma_{jl}-T_{ij}^r\right)
\end{equation}
and the compatibility condition (\ref{cond1}) can be written into the form
\begin{equation}
\label{cond2}
X_i^{h^*}F+\frac{1}{2}y^j \left(T^{l}_{jk}\gamma^{kr}\gamma_{il}+T^{l}_{ik}\gamma^{kr}\gamma_{jl}-T_{ij}^r\right)\circ \pi \frac{\partial F}{\partial y^r}=0\ \  (i=1, \ldots,n).
\end{equation}

Formula (\ref{Cproc}) shows that the correspondence $\nabla \rightleftharpoons T$ preserves the affine combinations of the linear connections, i.e. for any real number $\lambda\in \mathbb{R}$ we have 
$$\lambda \nabla_1+(1-\lambda)\nabla_2 \rightleftharpoons \lambda T_1+(1-\lambda)T_2.$$
Additionally, if $\nabla_1$ and $\nabla_2$ satisfy the compatibility condition (\ref{cond1}) then so does 
$$\lambda \nabla_1+(1-\lambda)\nabla_2.$$
This means that the set containing the restrictions of the torsion tensors of the compatible linear connections to the Cartesian product $\displaystyle{T_pM\times T_pM}$ is  an affine subspace in the linear space $\wedge^2 T_p^*M \otimes T_pM$ for any $p\in M$. As the point is varying we have an affine distribution of the torsion tensor bundle $\wedge^2 T^*M \otimes TM$. In terms of local coordinates, the bundle is spanned by
$$du^i \wedge du^j \otimes \frac{\partial}{\partial u^k} \quad (1\leq i < j\leq n, k=1, \ldots, n)$$
and its dimension is $\displaystyle{\binom{n}{2}n}$.

\begin{Def} The products 
$$du^i \wedge du^j \otimes \frac{\partial}{\partial u^k} \quad (1\leq i< j \leq n, k=1, \ldots, n)$$
form an orthonormal basis at the point $p\in M$ if the coordinate vector fields $\partial/\partial u^1, \ldots, \partial/\partial u^n$ form an orthonormal basis with respect to the averaged Riemannian metric at the point $p\in M$. The norm of the torsion tensor is defined by
\begin{equation}
\label{torsionnorm}
\|T_p\|^2=\sum_{1\leq i < j \leq n} \sum_{k=1}^n {T_{ij}^k(p)}^2
\end{equation}
provided that $\displaystyle{T=\sum_{1\leq i < j \leq n} \sum_{k=1}^nT_{ij}^k du^i \wedge du^j \otimes \frac{\partial}{\partial u^k}}$ 
and the products form an orthonormal basis at the point $p\in M$. The corresponding inner product is
$$\langle T_p, S_p \rangle=\sum_{1\leq i < j \leq n} \sum_{k=1}^n T_{ij}^k(p)S_{ij}^k(p).$$
\end{Def}

Let a point $p\in M$ be given and consider the affine subspace $A_p$ in $\wedge^2 T_p^*M \otimes T_pM$ defined by
\begin{equation}
\label{condextb} X_i^{h^*}F(v)+\frac{1}{2}y^j (v)\left(T^{l}_{jk}\gamma^{kr}\gamma_{il}+T^{l}_{ik}\gamma^{kr}\gamma_{jl}-T_{ij}^r\right)(p) \frac{\partial F}{\partial y^r}(v)=0,
\end{equation}
where $i=1, \ldots, n$ and $v\in T_pM$. Note that $A_p$ is nonempty because it contains the restrictions of the torsion tensors of the compatible linear connections to the Cartesian product $\displaystyle{T_pM\times T_pM}$. If $T_p\in A_p$ then we can write $A_p$ as the translate $\displaystyle{A_p=T_p+H_p}$, where the linear subspace $H_p \subset \wedge^2 T_p^*M \otimes T_pM$ is defined by
\begin{equation}
\label{condextc} \frac{1}{2}y^j (v)\left(T^{l}_{jk}\gamma^{kr}\gamma_{il}+T^{l}_{ik}\gamma^{kr}\gamma_{jl}-T_{ij}^r\right)(p) \frac{\partial F}{\partial y^r}(v)=0,
\end{equation}
where $i=1, \ldots, n$ and $v\in T_pM$.

\begin{Lem}
\label{lemmakey} If $\ \nabla$ is a compatible linear connection then the linear subspace $H_p$ is invariant under the action
$$(\varphi T)_p(v,w):=\varphi^{-1} T_p(\varphi(v), \varphi(v))$$
of the holonomy group of $\nabla$ at the point $p\in M$.
\end{Lem}

\begin{proof} Suppose that  the coordinate vector fields $\partial/\partial u^1, \ldots, \partial/\partial u^n$ form an orthonormal basis with respect to the averaged Riemannian metric at the point $p\in M$ and let $Q_i^j$ be the matrix representation of $\varphi$, $P_j^i=\left(Q_i^j\right)^{-1}$. Evaluating (\ref{condextc}) at $\varphi(v)$ we have an equivalent system of equations because $v$ runs through the non-zero elements in $T_pM$. So does $\varphi(v)$. Therefore (\ref{condextc}) is equivalent to
$$\frac{1}{2}y^j \circ \varphi (v)\left(T^{l}_{jk}\gamma^{kr}\gamma_{il}+T^{l}_{ik}\gamma^{kr}\gamma_{jl}-T_{ij}^r\right)(p) \frac{\partial F}{\partial y^r}\circ \varphi(v)=0,$$
$$ \frac{1}{2}y^b (v)\left(T^{l}_{jk}Q_b^j P_r^c \gamma^{kr}\gamma_{il}+T^{l}_{ik}Q_b^j P_r^c\gamma^{kr}\gamma_{jl}-T_{ij}^rQ_b^j P_r^c\right)(p) \frac{\partial F}{\partial y^c}\circ \varphi(v)=0,$$
where $i=1, \ldots, n$ and $v\in T_pM$. Indeed, $\displaystyle{\varphi^j=y^j\circ \varphi=Q_b^j y^b}$ and the invariance property $F\circ \varphi=F$ implies that 
$$\frac{\partial F}{\partial y^c}(v)=\frac{\partial F\circ \varphi}{\partial y^c}(v)=Q_c^r\frac{\partial F}{\partial y^r}\circ \varphi(v) \ \Rightarrow \ P_r^c\frac{\partial F}{\partial y^c}(v)=\frac{\partial F}{\partial y^r}\circ \varphi(v).$$
Since the identities $\displaystyle{Q_a^i Q_b^j\gamma_{ij}(p)=\gamma_{ab}(p)}$ and $\displaystyle{P_i^a P_j^b \gamma^{ij}(p)=\gamma^{ab}(p)}$ give that $\displaystyle{Q_b^j \gamma_{jl}(p)=P_l^j \gamma_{jb}(p)}$ and $\displaystyle{P_j^b \gamma^{jk}(p)=Q_j^k\gamma^{jb}(p)}$, we have  
$$\frac{1}{2}y^b (v)\left(T^{l}_{jk}Q_b^j Q_r^k \gamma^{rc}\gamma_{il}+T^{l}_{ik}Q_b^j Q_r^k\gamma^{rc}\gamma_{jl}-T_{ij}^r Q_b^j P_r^c\right)(p) \frac{\partial F}{\partial y^c}(v)=0,$$
where $i=1, \ldots, n$ and $v\in T_pM$. Taking the product by the matrix $Q^i_a$ the equivalent system of equations is
$$\frac{1}{2}y^b (v)\left(T^{l}_{jk}Q_a^i Q_b^j Q_r^k \gamma^{rc}\gamma_{il}+T^{l}_{ik}Q_a^i Q_b^j Q_r^k\gamma^{rc}\gamma_{jl}-T_{ij}^r Q_a^i Q_b^j P_r^c\right)(p) \frac{\partial F}{\partial y^c}(v)=0,$$
$$\frac{1}{2}y^b (v)\left(T^{l}_{jk}P_l^i Q_b^j Q_r^k \gamma^{rc}\gamma_{ia}+T^{l}_{ik}Q_a^i P_l^j Q_r^k\gamma^{rc}\gamma_{jb}-T_{ij}^r Q_a^i Q_b^j P_r^c\right)(p) \frac{\partial F}{\partial y^c}(v)=0,$$
$$\frac{1}{2}y^b (v)\left(\left(\varphi T\right)_{br}^i \gamma^{rc}\gamma_{ia}+\left(\varphi T\right)_{ar}^j \gamma^{rc}\gamma_{jb}-\left(\varphi T\right)_{ab}^c\right)(p) \frac{\partial F}{\partial y^c}(v)=0,$$
where $a=1, \ldots, n$ and $v\in T_pM$.
\end{proof}

\begin{Rem} {\emph{The previous argument is obviously working for any element of the group $G$ containing orthogonal transformations of the tangent space $T_pM$  with respect to the averaged Riemannian metric such that the Finslerian indicatrix is invariant: $F\circ \varphi=F$ ($\varphi\in G$). It is also clear that $G$ is a compact subgroup in the orthogonal group and $\textrm{Hol}_p \nabla \subset G$.}}
\end{Rem}

\begin{Cor} 
\label{smoothness} If we have a connected generalized Berwald manifold then the mapping $p\in M\to A_p\subset \wedge^2 T_p^*M \otimes T_pM$ is a smooth affine distribution of constant rank of the torsion tensor bundle.
\end{Cor}

\begin{proof} Let $\ \nabla$ be a compatible linear connection, $T$ be its torsion tensor and the point $p\in M$ be given. According to Lemma \ref{lemmakey} we have that $\displaystyle{A_q=T_q+\tau_{pq}(L_p)}$ for any $q\in M$, where $\tau_{pq}$ is the parallel transport along an arbitrary curve joining $p$ and $q$. 
\end{proof}

In what follows we introduce the extremal compatible linear connection of a generalized Berwald manifold in terms of its torsion $T^0$ by taking the closest point $T^0_q \in A_q$ to the origin for any $q\in M$.

\begin{Def} The extremal compatible linear connection of a generalized Berwald manifold is the uniquely determined compatible linear connection minimizing the norm of its torsion by taking the values of the pointwise minima. 
\end{Def}

\subsection{A conditional extremum problem for the extremal compatible linear connection} Let a point $p\in M$ be given and consider the conditional extremum problem
\begin{equation}
\label{condext} \min \frac{1}{2} \|T_p\|^2 \quad \textrm{subject to} \quad T_p \in A_p,
\end{equation}
where the affine subspace $A_p\subset \wedge^2 T_p^*M \otimes T_pM$ defined by (\ref{condextb}). First of all note that the coefficient of $T_{ab}^c$ ($1\leq a < b \leq n$, $c=1, \ldots, n$) is
\begin{equation}
\label{coeff}
\sigma_{c; i}^{ab}=\frac{1}{2}\left(\left(y^a \gamma^{br}-y^b\gamma^{ar}\right)\frac{\partial F}{\partial y^r}\gamma_{ic}+\left(\delta_i^a \gamma^{br}-\delta_i^b\gamma^{ar}\right)\frac{\partial F}{\partial y^r}y^j\gamma_{jc}-\left(\delta_i^a y^b-\delta_i^b y^a\right)\frac{\partial F}{\partial y^c}\right),
\end{equation}
where the index $i=1, \ldots, n$ refers to the corresponding equation in (\ref{cond2}). The symmetric differences in formula (\ref{coeff}) is due to $a< b$. If the coordinate vector fields $\displaystyle{\partial/\partial u^1, \ldots, \partial/\partial u^n}$ form an orthonormal basis at $p\in M$ with respect to the averaged Riemannian metric $\gamma$, then 
\begin{equation}
\label{coeffort}
\sigma_{c; i}^{ab}=\frac{1}{2}\left(\delta_i^c\left(y^a \frac{\partial F}{\partial y^b}-y^b\frac{\partial F}{\partial y^a}\right)+\delta_i^a \left(y^c\frac{\partial F}{\partial y^b}-y^b\frac{\partial F}{\partial y^c}\right)-\delta_i^b \left(y^c\frac{\partial F}{\partial y^a}-y^a\frac{\partial F}{\partial y^c}\right)\right).
\end{equation}
Since the vector fields 
$$y^a \frac{\partial }{\partial y^b}-y^b\frac{\partial }{\partial y^a}, \ y^c\frac{\partial}{\partial y^b}-y^b\frac{\partial}{\partial y^c}, \ y^c\frac{\partial}{\partial y^a}-y^a\frac{\partial}{\partial y^c}$$
come from the Liouville vector field (the outer unit normal to the Finslerian indicatrix) by an Euclidean quarter rotation in the corresponding $2$-planes, their actions on $F$ are automatically zero at the contact points of the Finslerian and the Riemannian spheres.

\subsection{Vertical and horizontal contact points}

\begin{Def}
\label{contactpoints} The nonzero element $v\in T_pM$ satisfying
$$\frac{\partial \log F}{\partial y^i}(v)=\frac{\partial \log F^*}{\partial y^i}(v) \ \ (i=1, \ldots, n)$$
is called a vertical contact point of the Finslerian and the averaged Riemannian metric functions. The nonzero element $v\in T_pM$ is a horizontal contact point of the Finslerian and the averaged Riemannian metric functions if
$$X_i^{h^*} F(v)=0 \ \ (i=1, \ldots, n).$$
The tangent space at $p\in M$ is vertical/horizontal contact if all nonzero elements $v\in T_pM$ are vertical/horizontal contact. 
\end{Def}

The vertical/horizontal contact vector fields can also be defined in a similar way: the vector field $X$ on the base manifold is vertical/horizontal contact if either $X(p)\in T_pM$ is vertical/horizontal contact or $X(p)={\bf 0}$.

\begin{Rem}
{\emph{First of all note that the vertical contact points are independent of the choice of the coordinate system. Geometrically,  the tangent hyperplanes of the Finslerian and the Riemannian spheres passing through a vertical contact point are the same in the corresponding tangent space. This is because their Euclidean gradient vectors in $T_pM$ are proportional:}}
\begin{equation}
\label{vertcont1}
\frac{\partial F}{\partial y^i}(v)=\frac{F}{F^*}(v)\frac{\partial F^*}{\partial y^i}(v) \ \ (i=1, \ldots, n),
\end{equation}
{\emph{where the coordinate vector fields $\displaystyle{\partial/\partial u^1, \ldots, \partial/\partial u^n}$ form an orthonormal basis at $p\in M$ with respect to the averaged Riemannian metric $\gamma$.}}
\end{Rem}

\begin{Cor}
\label{contactpoints01} The vertical contact points of a generalized Berwald manifold are horizontal contact points.
\end{Cor}

\begin{Pf} Suppose that the coordinate vector fields $\displaystyle{\partial/\partial u^1, \ldots, \partial/\partial u^n}$ form an orthonormal basis at $p\in M$ with respect to the averaged Riemannian metric $\gamma$. It can be easily seen that 
\begin{equation}
\label{vertcont2}
y^a \frac{\partial F^*}{\partial y^b}(v)-y^b\frac{\partial F^*}{\partial y^a}(v)=y^c\frac{\partial F^*}{\partial y^b}(v)-y^b\frac{\partial F^*}{\partial y^c}(v)=y^c\frac{\partial F^*}{\partial y^a}(v)-y^a\frac{\partial F^*}{\partial y^c}(v)=0.
\end{equation}
Therefore, by formula (\ref{vertcont1}), 
\begin{equation}
\label{vertcont3}
y^a \frac{\partial F}{\partial y^b}(v)-y^b\frac{\partial F}{\partial y^a}(v)=y^c\frac{\partial F}{\partial y^b}(v)-y^b\frac{\partial F}{\partial y^c}(v)=y^c\frac{\partial F}{\partial y^a}(v)-y^a\frac{\partial F}{\partial y^c}(v)=0.
\end{equation}
This means that $\sigma^{ab}_{c;i}(v)=0$ and the corresponding equations in (\ref{condextb}) reduce to $\displaystyle{X_i^{h^*} F (v)=0}$ for any $i=1, \ldots, n$.
\end{Pf}

\begin{Cor}
\label{contactpoints02}
Let $p\in M$ be a given point of a connected generalized Berwald manifold. If any nonzero element $v\in T_pM$ is a vertical contact point of the Finslerian and the averaged Riemannian metric functions, i.e. $T_pM$ is a vertical contact tangent space then we have a Riemannian manifold.
\end{Cor}

\begin{Pf}
If any nonzero element in $T_pM$ is a vertical contact point of the Finslerian and the averaged Riemannian metric functions then we have that 
$$\log \frac{F}{F^*} (v)=\textrm{const.}\ \ (v\in T_pM),$$
i.e. $\displaystyle{F(v)=e^{\textrm{const.}} F^*(v)}$ for any  nonzero element $v\in T_pM$. This means that the Finslerian indicatrix is a quadratic hypersurface at a single point $p\in M$. Since we have linear parallel transports between different tangent spaces, the same is true for the Finslerian indicatrix at any point of the (connected) base manifold, i.e. we have a Riemannian manifold.
\end{Pf}

\begin{Cor}
\label{contactpoints03} Let $p\in M$ be a given point of a generalized Berwald manifold. If any nonzero element $v\in T_pM$ is a horizontal contact point of the Finslerian and the averaged Riemannian metric functions, i.e. $T_pM$ is a horizontal contact tangent space then the torsion of the extremal compatible linear connection is identically zero at the point $p\in M$. 
\end{Cor}

\begin{Pf}
The statement is trivial because the zero element in $\wedge^2 T_p^*M \otimes T_pM$ solves the equations in (\ref{condextb}) under the conditions $X_i^{h^*} F(v)=0$ $(i=1, \ldots, n)$ for any nonzero $v\in T_pM$. 
\end{Pf}

\begin{Rem} 
\label{contactpoints04}{\emph{If we have a classical Berwald manifold then any nonzero element $v\in TM$ is a horizontal contact point of the Finslerian and the averaged Riemannian metric functions and vice versa. The extremal compatible linear connection is $\nabla^*$ (the L\'{e}vi-Civita connection of the averaged Riemannian metric) with vanishing torsion.}}
\end{Rem}

\subsection{The reference element method} In what follows we are going to find the Lagrange multipliers for the torsion tensor of the extremal compatible linear connection. They transform the compatibility condition to a system of linear equations containing at most $n$ unknown parameters instead of $\displaystyle{\binom{n}{2}n}$. Let a point $p\in M$ and the reference element $v\in T_pM\setminus \{{\bf 0}\}$ be given and consider the conditional extremum problem  
\begin{equation}
\label{condextref} \min \frac{1}{2} \|T_p\|^2 \ \ \ \textrm{subject to} \ \ \ T_p\in A_p(v); 
\end{equation}
the affine subspace $A_p(v)\subset \wedge^2 T_p^*M \otimes T_pM$ is defined by
\begin{equation}
\label{condextd} X_i^{h^*}F(v)+\frac{1}{2}y^j (v)\left(T^{l}_{jk}\gamma^{kr}\gamma_{il}+T^{l}_{ik}\gamma^{kr}\gamma_{jl}-T_{ij}^r\right)(p) \frac{\partial F}{\partial y^r}(v)=0
\end{equation}
where $i=1, \ldots, n$. It is clear that $\displaystyle{A_p = \bigcap_{v\in T_pM\setminus \{{\bf 0}\}} A_p(v) \subset A_p(v)}$.

\begin{Lem} Introducing the notations
$$g_i(T_p, v):=X_i^{h^*}F(v)+\frac{1}{2}y^j(v)\left(T^{l}_{jk}\gamma^{kr}\gamma_{il}+T^{l}_{ik}\gamma^{kr}\gamma_{jl}-T_{ij}^r\right)(p) \frac{\partial F}{\partial y^r}(v) \quad (i=1, \ldots, n),$$
$$\textrm{\ rank}\ \frac{\partial g_i}{\partial T_{bc}^a}(T_p, v)=\left\{
\begin{array}{rl}
n&\textrm{if $v$ is not a vertical contact point in $T_pM$}\\
0&\textrm{otherwise}.
\end{array}
\right.$$
\end{Lem}

\begin{proof} If $v$ is a vertical contact point of the Finslerian and the averaged Riemannian metric functions then $\sigma_{c; i}^{ab}(v)=0$ $(i=1, \ldots, n)$ because of (\ref{coeffort}) - (\ref{vertcont3}). Otherwise 
$$\begin{array}{|c|c|c|c|c|}
\hline
&&&&\\ 
a=c=i, & b=c=i, & a=c=i,&b=c=i,&a\neq i,\\ 
i< b, i=j&  a< i, i=j&  i< b, i\neq j&a < i, i\neq j&b\neq i, c\neq i\\ 
&&&&\\ \hline
&&&&\\
\sigma_{i;i}^{i b}=y^i \frac{\partial F}{\partial y^b}-y^b\frac{\partial F}{\partial y^i}&\sigma_{i;i}^{a i}=y^a \frac{\partial F}{\partial y^i}-y^i\frac{\partial F}{\partial y^a}&\sigma_{i;j}^{i b}=0&\sigma_{i;j}^{a i}=0&\sigma_{c; i}^{a b}=0\\ 
&&&&\\ \hline
\end{array}$$
and we can find at least one non-zero value among the coefficients $\sigma_{i;i}^{i b}(v)$, where $i< b$ and
$\sigma_{i;i}^{a i}(v)$, where $a<i$ in each row\footnote{Suppose, in contrary, that (for example) $i=1$ and $\sigma_{1;1}^{1 2}=\ldots=\sigma_{1;1}^{1 n}=0,$ i.e. $\displaystyle{v^1\frac{\partial F}{\partial y^b}(v)-v^b \frac{\partial F}{\partial y^1}(v)=0}$, where $b=2, \ldots, n$. Therefore
$$v^1\sum_{b=2}^n v^b \frac{\partial F}{\partial y^b}(v)-\frac{\partial F}{\partial y^1}(v)\sum_{b=2}^n \left(v^b\right)^2=0$$
and the homogeneity property of the metric function implies that 
$$v^1\left(F(v)-v^1\frac{\partial F}{\partial y^1}(v)\right)-\frac{\partial F}{\partial y^1}(v)\left(\left(F^*(v)\right)^2-\left(v^1\right)^2\right)=0 \ \ \Rightarrow\ \ \frac{1}{F(v)}\frac{\partial F}{\partial y^1}(v)=\frac{v^1}{\left(F^*(v)\right)^2}=\frac{1}{F^*(v)}\frac{\partial F^*}{\partial y^1}(v).$$}. The subsequent rows contain zeros in the corresponding positions:
$\displaystyle{\sigma_{i;i+1}^{i b}(v)=\sigma_{i;i+1}^{a i}(v)=\sigma_{i;i+2}^{i b}(v)=\sigma_{i;i+2}^{a i}(v)=\ldots=0}$.
\end{proof}

Let a point $p\in M$ of a connected generalized Berwald manifold be given.

\begin{itemize}
\item [(A)] If any nonzero element $v\in T_pM$ is a vertical contact point of the Finslerian and the averaged Riemannian metric functions, i.e. $T_pM$ is a vertical contact tangent space then we have a Riemannian manifold (see Corollary \ref{contactpoints02}) and the extremal compatible linear connection is $\nabla^*$.
\end{itemize}
Otherwise, let $p\in M$ be given and suppose that $v\in T_pM$ is not a vertical contact reference element. Using the notations in subsection 2.1 the conditional extremum problem can be written into the form
$$\min \frac{1}{2} \|T_p\|^2 \quad \textrm{subject to} \quad X_i^{h^*}F(v)+\langle T_p, \sigma_i(v) \rangle=0 \quad (i=1, \ldots, n),$$
where
$$\sigma_i=\sum_{1\leq a < b \leq n} \sum_{c=1}^n \sigma_{c; i}^{ab}\frac{\partial}{\partial y^a}\wedge \frac{\partial}{\partial y^b} \otimes dy^c$$
and the inner product on the vertical subspace $V_v \wedge^2 T^*_pM \otimes T_pM$ is induced by the Riemannian metric of the torsion tensor bundle:
$$\langle T_p, \sigma_i(v) \rangle =\sum_{1\leq a < b \leq n} \sum_{c=1}^n T_{ab}^c(p) \sigma_{c; i}^{ab}(v);$$
note that the torsion tensor $T$ is identified with its vertically lifted tensor if necessary, i.e. 
$$T=\sum_{1\leq a < b \leq n} \sum_{c=1}^n T_{ab}^c \frac{\partial}{\partial u^a}\wedge \frac{\partial}{\partial u^b} \otimes du^c \rightleftharpoons \sum_{1\leq a < b \leq n} \sum_{c=1}^n T_{ab}^c\circ \pi  \frac{\partial}{\partial y^a}\wedge \frac{\partial}{\partial y^b} \otimes dy^c.$$
The Lagrange method says that 
$$\frac{\mathcal{\partial L}(v)}{\partial T_{ab}^c}(T_p, \lambda_1(v), \ldots, \lambda_n(v))=0,$$
where
$$\mathcal{L}(v)(T_p, \lambda_1(v), \ldots, \lambda_n(v))=\frac{1}{2}\|T_p\|^2-\sum_{j=1}^{n} \lambda_{j}(v)g_{j}(T_p, v)=$$
$$\frac{1}{2}\|T_p\|^2-\sum_{j=1}^{n} \lambda_{j}(v)\left(X_j^{h^*}F(v)+\langle T_p, \sigma_j(v) \rangle \right),$$
i.e.
$$T_{ab}^c(p)-\sum_{j=1}^{n} \lambda_{j}(v)\sigma_{c; j}^{ab}(v)=0 \ \ \Rightarrow \ \ T_{ab}^c(p)=\sum_{j=1}^{n} \lambda_{j}(v)\sigma_{c; j}^{ab}(v).$$
Subtituting into the conditional equations we have that 
\begin{equation}
\label{multiplierseq}
X_{i}^{h^*}F(v)+\sum_{j=1}^n \lambda_{j}(v)\sum_{1\leq a < b \leq n} \sum_{c=1}^n \sigma_{c; i}^{ab}(v)\sigma_{c; j}^{ab}(v)=0 \quad (i=1, \ldots, n)
\end{equation}
and 
\begin{equation}
\label{formula01}
-\mathcal{G}^{-1}\left(\sigma_1(v), \ldots, \sigma_n(v)\right)\left(X_{1}^{h^*}F(v), \ldots, X_n^{h^*}F(v)\right)^T=\left(\lambda_1(v), \ldots, \lambda_n(v)\right)^T,
\end{equation}
where $i=1, \ldots, n$ and $\displaystyle{\mathcal{G}\left(\sigma_1(v), \ldots, \sigma_n(v)\right)}$ is the Gramian with respect to the inner product
$$\langle \sigma_i(v), \sigma_j(v)\rangle=\sum_{1\leq a < b \leq n} \sum_{c=1}^n \sigma_{c; i}^{ab}(v)\sigma_{c; j}^{ab}(v)$$
on the vertical subspace $V_v \wedge^2 T_p^*M \otimes T_pM$. The uniquely determined solutions $\lambda_{1}(v), \ldots, \lambda_n(v)$ give the closest element $T_p^0(v)\in A_p(v)$ to the origin in the linear space $\wedge^2 T_p^*M \otimes T_pM$. Therefore
$$A_p(v)=T_p^0(v)+\mathcal{L}^{\bot}\left(\sigma_1(v), \ldots, \sigma_n(v)\right),$$
where
\begin{equation}
\label{formula02}
T_p^0(v)=\sum_{j=1}^n \lambda_j(v)\sigma_j(v)
\end{equation}
and the coefficients are given by formula (\ref{formula01}). Since  
$$A_p=\bigcap_{v\ \in \ T_pM\setminus \{{\bf 0}\} \ \textrm{is not a vertical contact point}}T_p^0(v)+\mathcal{L}^{\bot}\left(\sigma_1(v), \ldots, \sigma_n(v)\right)$$
it follows that 
$$\dim A_p \leq \dim A_p(v)=\binom{n}{2}n-n,$$
where $v$ is not a vertical contact element. Note that the intersection can be taken with respect to Finslerian (or Riemannian) unit vectors because of the homogeneity properties:
$$A_p=\bigcap_{v\ \in \ \partial K_p \ \textrm{is not a vertical contact point}}T_p^0(v)+\mathcal{L}^{\bot}\left(\sigma_1(v), \ldots, \sigma_n(v)\right)=$$
$$\ \ \ \ \ \ \bigcap_{v\ \in \ \partial K^*_p \ \textrm{is not a vertical contact point}}T_p^0(v)+\mathcal{L}^{\bot}\left(\sigma_1(v), \ldots, \sigma_n(v)\right).$$
\begin{itemize}
\item [(B)] In case of a horizontal but not vertical contact reference element $T_p^0(v)={\bf 0}\in \wedge^2 T_p^*M \otimes T_pM$ because of Definition \ref{contactpoints} and formula (\ref{formula01}). If any nonzero element $v\in T_pM$ is a horizontal contact point of the Finslerian and the averaged Riemannian metric functions, i.e. $T_pM$ is a horizontal contact tangent space then the solution of the extremum problem (\ref{condextref}) is $T_p^0:={\bf 0}$ independently of the reference elements (see Corollary \ref{contactpoints03}). 
\item [(C)] Otherwise, let $p\in M$ be given and suppose that $v\in T_pM$ is not a horizontal contact reference element, i.e. $T_p^0(v)\neq {\bf 0}$. Since $A_p\subset A_p(v)$ it follows that $A_p$ consisting of the solutions of the conditional equation without any reference element must be contained in the hyperplane
\end{itemize}
\begin{equation}
\label{hyperplane}
\langle T_p^0(v), T_p-T_p^0(v)\rangle=0
\end{equation}
of dimension $\displaystyle{\binom{n}{2}n-1}$ in $\wedge^2 T_p^*M \otimes T_pM$. Therefore the process in case (C) can be completed as follows:
\begin{itemize}
\item [\textrm{Step 1}] Let $T^0_p(v_1)$ be the uniquely determined, not identically zero solution of the conditional extremum problem 
\begin{equation}
\label{condext:01} \min \frac{1}{2} \|T_p\|^2 \quad \textrm{subject to} \quad T_p\in A_p(v_1),
\end{equation}
\end{itemize}
where $v_1\in T_pM$ is not a horizontal contact point. If $g_i(T^0_p(v_1), v)=0$ ($i=1, \ldots, n$) for any $v\in T_pM$ then $T^0_p=T_p^0(v_1)$ and we are done.
\begin{itemize}
\item [\textrm{Step 2}] Otherwise, let us choose a nonzero element $v_2$ such that $g_i(T^0_p(v_1), v_2)\neq 0$ for at least one of the indices $i=1, \ldots, n$. Taking $T_p^0(v_1)$ as the origin, let $T_p^0(v_1, v_2)$ be the uniquely determined, not identically zero solution of the conditional extremum problem
\begin{equation}
\label{condext:02} \min \frac{1}{2} \|T_p\|^2 \quad \textrm{subject to} \quad T_p \in A_p(v_2) 
\end{equation}
and
\begin{equation}
\label{extracond:01}
\langle T_p^0(v_1), T_p-T_p^0(v_1) \rangle=0.
\end{equation}
\end{itemize}
Geometrically, $T_p^0(v_1, v_2)$ is the orthogonal projection of $T_p^0(v_1)$ onto $A_p(v_2)$ in the hyperplane (\ref{extracond:01}) of dimension $\displaystyle{\binom{n}{2}n-1}$. If
$$\mathcal{G}\left(\sigma_1(v_2), \ldots, \sigma_n(v_2), T_p^0(v_1)\right)=0$$
then $T_p^0(v_1, v_2)=T_p^0(v_2)$. Otherwise
\begin{equation}
\label{formula03}
-\mathcal{G}^{-1}\left(\sigma_1(v_2), \ldots, \sigma_n(v_2), T_p^0(v_1)\right)\left(X_{1}^{h^*}F(v_2), \ldots, X_n^{h^*}F(v_2), -\|T_p^0(v_1)\|^2\right)^T=
\end{equation}
$$\left(\lambda_1(v_2), \ldots, \lambda_n(v_2), \lambda_{n+1}(v_2)\right)^T.$$
Especially, equation (\ref{extracond:01}) is equivalent to 
$$\lambda_{n+1}(v_2)=1-\sum_{j=1}^n \lambda_j(v)\frac{\langle \sigma_j(v_2), T^0_p(v_1)\rangle}{\|T^0_p(v_1)\|^2}.$$
If $g_i(T_p^0(v_1, v_2), v)=0$ ($i=1, \ldots, n$) for any $v\in T_pM$ then $T^0_p=T_p^0(v_1, v_2)$ and we are done.
\begin{itemize}
\item [\textrm{Step 3}] Otherwise, let us choose a nonzero element $v_3$ such that $g_i(T^0_p(v_1, v_2), v_3)\neq 0$ for at least one of the indices $i=1, \ldots, n$. Taking $T_p^0(v_1, v_2)$ as the origin, let $T_p^0(v_1, v_2, v_3)$ be the uniquely determined, not identically zero solution of the conditional extremum problem 
\begin{equation}
\label{condext:03} \min \frac{1}{2} \|T_p\|^2 \quad \textrm{subject to} \quad T_p\in A_p(v_3)
\end{equation}
and
\begin{equation}
\label{extracond:02}
\langle T_p^0(v_1), T_p-T_p^0(v_1, v_2)\rangle=0, \ \langle T_p^0(v_1, v_2), T_p-T_p^0(v_1, v_2)\rangle=0.
\end{equation}
\end{itemize}
Geometrically, $T_p^0(v_1)$, $T_p^0(v_1, v_2)$ and $T_p^0(v_1, v_2, v_3)$ form an orthogonal chain in the sense that
$$ T_p^0(v_1) \  \bot\ \ T_p^0(v_1, v_2)-T_p^0(v_1) \ \textrm{and} \ \ T_p^0(v_1), \ T_p^0(v_1, v_2)-T_p^0(v_1) \ \bot \ T_p^0(v_1, v_2, v_3)-T_p^0(v_1, v_2).$$
Therefore $T_p^0(v_1, v_2, v_3)$ is the element of a hyperplane of dimension $\displaystyle{\binom{n}{2}n-2}$. In general, if we have $T_p^0(v_1), \ldots, T_p^0(v_1, \ldots, v_{m-1})$ then
\begin{itemize}
\item $T_p^0=T_p^0(v_1, \ldots, v_{m-1})$ in case of $g_i(T_p^0(v_1, \ldots, v_{m-1}),v)=0$ ($i=1, \ldots, n$) for any $v\in T_pM$.
\item Otherwise, let us choose a nonzero element $v_m$ such that $g_i(T_p^0(v_1, \ldots, v_{m-1}), v_m)\neq 0$ for at least one of the indices $i=1, \ldots, n$. Taking $T_p^0(v_1, \ldots, v_{m-1})$ as the origin, let $T_p^0(v_1, \ldots, v_m)$ be the uniquely determined, not identically zero solution of the conditional extremum problem
\begin{equation}
\label{condext:0m} \min \frac{1}{2} \|T_p\|^2 \quad \textrm{subject to} \quad T_p\in A_p(v_m)
\end{equation}
and
\begin{equation}
\label{extracond:0m2}
\langle T_p^0(v_1), T_p-T_p^0(v_1, \ldots, v_{m-1})\rangle=0, \ \langle T_p^0(v_1, v_2), T_p-T_p^0(v_1, \ldots, v_{m-1})\rangle=0, \ \ldots,
\end{equation} 
$$\langle T_p^0(v_1, \ldots, v_{m-1}), T_p-T_p^0(v_1, \ldots, v_{m-1})\rangle=0.$$
\end{itemize}

\begin{Cor}
\label{alg}
For any $m\leq l$
$$T_p^0(v_1), \ T_p^0(v_1, v_2)-T_p^0(v_1), \ \ldots, $$
$$T_p^0(v_1, \ldots, v_{m-1})-T_p^0(v_1, \ldots, v_{m-2}) \ \bot \ T_p^0(v_1, \ldots,v_m)-T_p^0(v_1, \ldots, v_{m-1}),$$
i.e. $T_p^0(v_1), \ldots, T_p^0(v_1, \ldots,v_m)$ is an orthogonal chain, where $\displaystyle{l\leq \binom{n}{2}n}$ is its maximal length and $T_p^0=T_p^0(v_1, \ldots, v_l)$ is the torsion tensor of the extremal compatible linear connection. 
\end{Cor}

\begin{Thm} Let $M$ be a connected manifold equipped with a Finslerian metric $F$. It is a generalized Berwald manifold if and only if it is a classical Berwald manifold including the case of the Riemannian manifolds or 
\begin{itemize}
\item the vertical contact elements of the tangent manifold $TM$ are horizontal contact,
\item for any non-horizontal contact tangent space $T_pM$, the form 
$$T^0_p(v_1, \ldots, v_l)\in \wedge^2 T_p^*M \otimes T_pM$$
is independent of the choice of the non-horizontal contact elements, where $\displaystyle{l\leq \binom{n}{2}n}$ is the maximal length of the orthogonal chain $\displaystyle{T_p^0(v_1), T_p^0(v_1, v_2), \ldots, T_p(v_1, v_2, \ldots, v_l)}$, 
\item $\displaystyle{T^0\colon p\in M\to T^0_p:=\left\{\begin{array}{rl}
{\bf 0}&\textrm{in case of a horizontal contact tangent space,}\\
T^0_p(v_1, \ldots, v_l)&\textrm{otherwise}
\end{array}
\right.}$ is a continuous section of the torsion tensor bundle and the corresponding linear connection is compatible to the Finslerian metric.
\end{itemize}
\end{Thm}

\begin{proof}
Suppose that we have a connected generalized Berwald manifold. If it is not a classical Berwald manifold including the case of the Riemannian manifolds then we can refer to Corollary \ref{contactpoints01}. On the other hand, 
$$T^0\colon p\in M\to T^0_p:=\left\{\begin{array}{rl}
{\bf 0}&\textrm{in case of a horizontal contact tangent space,}\\
T^0_p(v_1, \ldots, v_l)&\textrm{otherwise}
\end{array}
\right.$$
is the torsion of the extremal compatible linear connection. Using Corollary \ref{smoothness}, it is a continuous section of the torsion tensor bundle. Conversely, suppose that we have a non-Riemannian and non-classical Berwald Finsler manifold such that the vertical contact elements of the tangent manifold $TM$ are horizontal contact. Then the compatibility equations are satisfied at each vertical contact element. Otherwise we can get the algorithm started by formulating the conditional extremum problem (\ref{condextref}) for any non vertical contact $v\in TM\setminus \{{\bf 0}\}$. The algorithm gives a uniquely determined output due to the independency of the form 
$$T^0_p(v_1, \ldots, v_l)\in \wedge^2 T_p^*M \otimes T_pM$$
of the choice of the non horizontal contact elements $v_1, \ldots, v_l\in T_pM$ ($p\in M$). Finally, if the section is continuous then we have a compatible linear connection to the Finslerian metric with continuous connection coefficients because of (\ref{torsion}). Using parallel translations with respect to such a connection we can conclude that the Finsler metric is monochromatic, i.e. there exists a linear mapping preserving the Finslerian norm of tangent vectors between the tangent spaces $T_pM$ and $T_qM$ for any pair of points $p, q\in M$. By the fundamental result of the theory \cite{BM} it is sufficient and necessary for a Finsler metric to be a generalized Berwald metric.
\end{proof}

\section{Summary} The idea of the paper is to find a distinguished compatible linear connection of a (connected) generalized Berwald manifold. Generalized Berwald manifolds are Finsler manifolds admitting linear connections such that the parallel transports preserve the Finslerian length of tangent vectors (compatibi\-li\-ty condition). By the fundamental result of the theory \cite{V5} such a linear connection must be metrical with respect to the averaged Riemannian metric given by integration of the Riemann-Finsler metric on the indicatrix hypersurfaces. The averaged Riemannian metric induces Riemannian metrics on the torsion tensor bundle $\wedge^2 T^*M \otimes TM$ and its vertical bundle in a natural way. Therefore we are looking for the extremal compatible linear connection in the sense that we want to minimize the norm of the torsion point by point.  It is a conditional extremum problem involving functions defined on a local neighbourhood of the tangent manifold. In case of a given point of the manifold, the reference element method helps us to start the application of the Lagrange multiplier rule. The finiteness of the rank of the torsion tensor bundle provides that we can compute the minimizer independently of the reference elements in finitely many steps at each point of the manifold. The pointwise solutions constitute a continuous section of the torsion tensor bundle. The continuity of the components of the torsion tensor implies the continuity of the connection parameters. Using parallel translations with respect to such a connection we can conclude that the Finsler metric is monochromatic. By the fundamental result of the theory \cite{BM} it is sufficient and necessary for a Finsler metric to be a generalized Berwald metric. Therefore we have an intrinsic algorithm to check the existence of compatible linear connections on a Finsler manifold because it is equivalent to the existence of the extremal compatible linear connection. 

\end{document}